\documentclass[12pt,a4paper]{amsart}

\usepackage[OT1]{fontenc}    %
\usepackage{type1cm}         %
\usepackage{amsthm}          %
\usepackage{amsmath}
\usepackage{amssymb}
\usepackage[dvips]{graphicx} 
\usepackage[bf]{caption}     
\usepackage{psfrag}          
\usepackage{geometry}        
\usepackage{version}         

\numberwithin{equation}{section}

\theoremstyle{plain}
\newtheorem{theorem}{Theorem}[section]
\newtheorem{corollary}[theorem]{Corollary}
\newtheorem{proposition}[theorem]{Proposition}
\newtheorem{lemma}[theorem]{Lemma}

\theoremstyle{remark}
\newtheorem{remark}[theorem]{Remark}

\theoremstyle{definition}
\newtheorem{definition}[theorem]{Definition}

\newcommand{\HH}{\mathcal{H}}

\newcommand{\EE}{\mathbf{E}}
\newcommand{\PP}{\mathbf{P}}
\newcommand{\QQ}{\mathcal{Q}}
\newcommand{\CC}{\mathcal{C}}
\newcommand{\XX}{\mathcal{X}}
\newcommand{\YY}{\mathcal{Y}}
\newcommand{\R}{\mathbb{R}}

\newcommand{\eps}{\varepsilon}

\DeclareMathOperator{\Cor}{Cor}
\DeclareMathOperator{\dimb}{dim_B}
\DeclareMathOperator{\udimb}{\overline{dim}_B}
\DeclareMathOperator{\ldimb}{\underline{dim}_B}

\DeclareMathOperator{\dimh}{dim_H}

\DeclareMathOperator{\dist}{dist}

\begin{document}

\title[Visible parts of fractal percolation]
{Visible parts of fractal percolation}

\author[I. Arhosalo]{Ida Arhosalo$^1$}
\address{Department of Mathematics and Statistics,  P.O. Box 35,
         FIN-40014 University of Jyv\"askyl\"a, Finland$^1$}
\email{idmaarho@jyu.fi$^1$}

\author[E. J\"arvenp\"a\"a]{Esa J\"arvenp\"a\"a$^2$}
\address{Department of Mathematical Sciences, P.O. Box 3000,
         90014 University of Oulu, Finland$^{2,3}$}
\email{esa.jarvenpaa@oulu.fi$^2$}

\author[M. J\"arvenp\"a\"a]{Maarit J\"arvenp\"a\"a$^3$}
\email{maarit.jarvenpaa@oulu.fi$^3$}

\author[M. Rams]{Micha\l\,Rams$^4$}
\address{Institute of Mathematics, Polish Academy of Sciences, 00-956 Warsaw,
         Poland$^4$}
\email{rams@impan.pl$^4$}

\author[P. Shmerkin]{Pablo Shmerkin$^5$}
\address{School of Mathematics, Alan Turing Building, University of Manchester,
         M13 9PL, UK$^5$}
\email{Pablo.Shmerkin@manchester.ac.uk$^5$}


\thanks{The authors acknowledge the support of the Centre of Excellence in
Analysis and Dynamics Research funded by the Academy of Finland. P.S. also
acknowledges support from EPSRC grant EP/E050441/1 and the University of
Manchester. }

\begin{abstract} We study dimensional properties of visible parts of
fractal percolation in the plane. Provided that the dimension of the fractal
percolation is at least 1, we show that, conditioned on non-extinction,
almost surely all visible parts from lines are 1-dimensional. Furthermore,
almost all of them have positive and finite Hausdorff measure. We also verify
analogous results for visible parts from points. These results are motivated
by an open problem on the dimensions of visible parts, see \cite{M2}.
\end{abstract}

\maketitle

\section{Introduction, notation and results}

\subsection{Visible parts}

The visible part of a compact set $E\subset \R^2$ from an affine line $\ell$
consists of those points $x\in E$ where one first hits the set
$E$ when looking perpendicularly from $\ell$. More precisely:

\begin{definition}\label{visibilitydef}
Let $E\subset\mathbb R^2$ be compact and let $\ell$ be an
affine line not meeting $E$. The \textbf{visible part} $V_\ell (E)$ of $E$ from
$\ell$ is
\[
V_\ell(E)=\{a\in E : [a,\Pi_\ell(a)]\cap E=\{a\}\}
\]
where $\Pi_\ell(a)$ is the projection of $a$ onto $\ell$ and
$[a,\Pi_\ell(a)]$ is the closed line segment joining $a$ to $\Pi_\ell(a)$.
Moreover, the visible part $V_x(E)$ of $E$ from a point
$x\in\mathbb R^2\setminus E$ is
\[
V_x(E)=\{a\in E : [a,x]\cap E=\{a\}\}.
\]
\end{definition}

In this paper we restrict our consideration to the planar case. Clearly,
Definition~\ref{visibilitydef} can be extended in a natural way to higher
dimensions, see \cite{JJMO}.
For a measure theoretic definition of visibility and related topics,
see \cite{Cs} and \cite{M2}.

The question of how the Hausdorff dimension, $\dimh$, of visible parts
depends on that of the original set has been considered in \cite{JJMO} and
\cite{O}.
In general, only ``almost all'' type of results are possible since there may be
exceptional directions, for example in the case of fractal graphs, see
\cite{JJMO}. Let $\mathcal L^n$ be the Lebesgue measure on $\mathbb R^n$.
There is a natural
Radon measure  $\Gamma$ on the space $\mathcal A$
of affine lines in the plane, that is, for all $A\subset\mathcal A$
\[
\Gamma(A)=\int\mathcal L^1(\{a\in L^\perp : L+a\in A\})
\,d\gamma(L),
\]
where $L$ is a line that goes through the origin, $L^\perp$ is the
orthogonal complement of $L$ and
$\gamma$ is the natural Radon measure on the space of all lines that go
through the
origin. Since every line through the origin can be parametrised by the angle
which it makes with the positive $x$-axis, the Lebesgue
measure $\mathcal L^1$ on the half open interval $[0,\pi)$ induces $\gamma$.

Let $E\subset\mathbb R^2$ be a compact set. The results in \cite{JJMO} for
dimensional properties of visible parts resemble the
Marstrand-Kaufman-Mattila -type projection results, according to which
\begin{equation}\label{projresult}
\dimh \Pi_L(E)=\min\{\dimh E,1\}
\end{equation}
for $\gamma$-almost all lines $L$ that go through the origin \cite{M1}.
For visible parts we have: if $\dimh E\le 1$ then
\begin{equation}\label{generalvisibility1}
\dimh V_\ell(E)=\dimh E\text{ and }\dimh V_x(E)=\dimh E
\end{equation}
for $\Gamma$-almost all affine lines $\ell$ not meeting $E$ and for
$\mathcal L^2$-almost all $x\in\mathbb R^2\setminus E$. On the other hand,
if $\dimh E>1$, then
\begin{equation}\label{generalvisibility2}
1\le\dimh V_\ell(E)\text{ and }1\le\dimh V_x(E)
\end{equation}
for  $\Gamma$-almost all affine lines $\ell$ not meeting $E$ and for
$\mathcal L^2$-almost all $x\in\mathbb R^2\setminus E$. These
results can be extended to higher dimensions by replacing 1 with $n-1$, see
\cite{JJMO}.

The methods utilised in \cite{JJMO} for proving \eqref{generalvisibility1}
and \eqref{generalvisibility2} are based on the generalized projection
formalism for parametrised families of transversal mappings due to Y. Peres and
W. Schlag \cite{PS}. The asymmetry between \eqref{projresult} and
\eqref{generalvisibility2} in the case $\dimh E>1$ is due to the following: in
\eqref{projresult} the upper
bound $\dimh \Pi_L(E)\le 1$ is trivial since $\Pi_L(E)$ is a subset of a line.
However, $V_\ell(E)$ does not have this restriction and a priori its dimension
could be as large as the dimension of $E$ (and indeed this can be the case, at
least for exceptional lines, as in the already mentioned example of fractal
graphs.)

The validity of the reverse inequality of \eqref{generalvisibility2} in general
is an open problem. In \cite{JJMO} it was verified for some concrete examples,
including quasi-circles and certain self-similar sets. In the planar case a
partial answer was given by T. C. O'Neil in \cite{O}.
Using energies, he showed that if a compact connected plane set $E$ has
Hausdorff dimension strictly larger than one, then visible parts from
almost all points
have Hausdorff dimension strictly less than the Hausdorff dimension of $E$.
In fact, for $\mathcal L^2$-almost all $x\in\mathbb R^2\setminus E$,
\[
\dimh V_x(E)\le\frac12+\sqrt{\dimh E-\frac34}.
\]

It is easy to see that $1$ is the only possible universal value for
Hausdorff dimension of typical visible parts of sets $E$ with $\dimh E>1$.
More precisely, if for all compact sets $E\subset\mathbb R^2$ with
$\dimh E>1$ there exists a constant $c$ such that $\dimh V_\ell(E)=c$
for almost all $\ell$, then $c=1$, see \cite {JJMO}.
In this paper we verify that this constancy result holds, in a strong form,
for typical random sets in fractal percolation.

\subsection{Fractal percolation} Fractal percolation is a natural model of
fractal sets that display stochastic self-similarity. Much is known about its
geometric properties, see
\cite{C} and \cite{G} and the references therein. We address the question of
studying dimensional properties of visible parts of
fractal percolation in the plane. It turns out that the reverse inequality in
\eqref{generalvisibility2} holds for \textit{all} lines almost surely
conditioned on non-extinction, in a strong quantitative form.
Moreover, the visible parts from almost every line have positive and finite
1-dimensional Hausdorff measure.  We underline that the methods we use
are different from those in \cite{JJMO} and \cite{O}.
Before stating the results, we recall the construction  of fractal percolation
and discuss some of its basic properties.

Fix $0<p<1$. We construct a random compact set as follows:
Let $Q_0=[0,1]\times [0,1]\subset\mathbb R^2$ be the unit square. Divide
$Q_0$ into four subsquares of equal size each of which
is chosen with probability $p$ and dropped with
probability $1-p$, independently of each other.
Denote by $\mathcal C_1$ the
collection of all chosen subsquares. For each
$Q\in\mathcal C_1$, we continue the same process by dividing $Q$ into four
subsquares of equal size. Again each of these subsquares is chosen with
probability $p$ and dropped with probability $1-p$, independently of each
other. The set of all chosen
squares at the second level is denoted by $\mathcal C_2$. Repeating
this process inductively gives the limiting random set $E$, defined as
\[
E = \bigcap_{n=1}^\infty \bigcup \{Q:Q\in\mathcal C_n\}.
\]
The probability space $\Omega$ is the space of all constructions and
the natural probability measure on $\Omega$ induced by this procedure is
denoted by $\PP$.

In \cite{CCD} J.T. Chayes, L. Chayes and R. Durrett verified that there is
a critical probability $0<p_c<1$ such that if $p<p_c$, then with probability
one $E$ is totally disconnected, whereas the opposing
sides of $Q_0$ are connected with positive probability provided that
$p>p_c$. This phenomenon is commonly referred to as fractal percolation.

We review some of the most basic facts on fractal percolation, and refer the
reader to \cite{G} or to \cite{C} for further background. Clearly, if $p<1$,
then there is a positive probability that the limit set $E$ is empty. A more
subtle question is for which values of $p$ the set $E$ is empty almost surely.
It turns out that
\[
\PP(E=\emptyset)=1\text{ if and only if }p\le\frac 14.
\]
Moreover, conditioned on non-extinction, that is $E\neq\emptyset$, we have
\[
\dimh E=\frac{\log(4p)}{\log2}
\]
almost surely. This implies that, conditioned on non-extinction, $\dimh E>1$
almost surely provided that $p>\frac 12$. In particular, when considering
dimensional properties of visible parts of $E$, we may restrict our
consideration to the case $p>\frac 12$, as the case $\frac 14<p\le\frac 12$ is
covered by the general equation \eqref{generalvisibility1}.

\begin{remark}\label{M2squares}
Instead of working with base 2 in the definition of fractal percolation one
could work with base $M$ for $M\ge2$, i.e. divide each square into $M^2$
subsquares of equal size and choose each of them with probability $p$ and
drop with probability $1-p$, independently of each other.
It is straightforward to see that all the results of this paper remain true
also in this case (with the threshold $p=\frac 12$ replaced by $p=\frac 1M$).
For notational simplicity we restrict our consideration to the case $M=2$.
\end{remark}

\subsection{Statement of results}

For a positive integer $k$, let $N_k(A)$ be the number of dyadic squares
of side length $2^{-k}$ that intersect a set $A\subset\mathbb R^2$. Recall that
the upper box dimension of a compact set $A$ is given by
\[
\overline{\dim}_B A=\limsup_{k\to\infty}\frac{\log N_k(A)}{\log 2^k}.
\]
Likewise one defines lower box dimension, and one says that the box dimension
exists, and is denoted by $\dim_B A$, if the lower and upper versions coincide.
We denote the 1-dimensional Hausdorff measure by $\mathcal H^1$. We now state
our main results.

\begin{theorem}\label{alllines}
Let $p>\frac12$. Conditioned on non-extinction, almost surely
\[
\dimh V_\ell(E)=\dimb V_\ell(E)=1
\]
for all lines $\ell$ not meeting $E$. Moreover, for any sequence $\{ a_k\}$
such that $\frac{a_k}k\rightarrow\infty$, one has almost surely that
\begin{equation} \label{growthsizecovering}
N_k(V_\ell(E))\le a_k 2^k
\end{equation}
simultaneously for all lines $\ell$ not meeting $E$ for all $k\ge K$.
Here $K$ depends on $E$, $\ell$ and the sequence $a_k$.
\end{theorem}

\begin{remark}\label{sizeofk}
For any closed $D\subset S^1$ with $D\cap\{(\pm 1,0),(0,\pm 1)\}=\emptyset$
one can choose uniform $K$ in \eqref{growthsizecovering} for all $\ell$
with $\ell\cap Q_0=\emptyset$ and $\theta(\ell)\in D$, where $\theta(\ell)$
is the angle between $\ell^\perp$ and the $x$-axis.
\end{remark}

We are also able to show that visible parts from a given line typically have
positive and finite length:

\begin{theorem}\label{maintheorem}
Let $\ell$ be any fixed line. Assume
that $p>\frac 12$. Then
\[
0<\mathcal H^1(V_\ell(E))<\infty
\]
almost surely conditioned on non-extinction and $E\cap\ell \neq\emptyset$.
\end{theorem}

As an immediate consequence of Theorem~\ref{maintheorem} we have:

\begin{corollary}\label{aboveperco}
Let $p>\frac12$. Conditioned on non-extinction, almost surely
\[
0<\mathcal H^1(V_\ell(E))<\infty
\]
for almost all lines $\ell$ which do not meet the unit square.
\end{corollary}

We do not know whether the exceptional set
$\{E:\mathcal H^1(V_\ell(E))=\infty\}$ in Theorem~\ref{maintheorem} depends on
$\ell$.

Above results concern visible parts from lines. Similar results are available
for visible parts from points; see Theorems~\ref{allpoints} and \ref{points}
 in Section~\ref{pointssection}.

\subsection{Notation and organization}

We henceforth fix a value of $p\in (\frac 12,1)$ for the rest of the paper.
We will use the $O(\cdot), \Omega(\cdot)$ notation: if $x,y$ are two positive
quantities, by $x=O(y)$ we mean that $x\leq C y$ for some constant $C$, and by
$x=\Omega(y)$ we mean $y=O(x)$. The implicit constant may depend only on $p$.
In particular, if the quantities $x,y$ are related to a stage $n$ of the
construction of fractal percolation, then the implicit constant is independent
of $n$.

The paper is organized in the following manner: in the next section we
verify crucial technical lemmas, in Section~\ref{linessection} we prove our
main theorems concerning visible parts from lines, and in the last section we
study visible parts from points.

\section{Technical lemmas}

In this section we verify some lemmas needed in the proof of our main theorems.
We start by showing that in Theorems~\ref{alllines} and \ref{maintheorem} it is
enough to consider lines that do not meet the closed unit square $Q_0$. For all
positive integers $n$, we will denote the set of all dyadic subsquares
of $Q_0$ of side length $2^{-n}$ by $\mathcal Q_n$. Recall that $\mathcal{C}_n$
is the random subset of $\mathcal Q_n$ consisting of the chosen squares of
side length $2^{-n}$. Throughout the paper, by a square we mean a closed dyadic
square with sides parallel to the axes.

\begin{lemma} \label{enoughnotmeetingsquare}
In Theorem~\ref{alllines}, it is enough to prove the statement for all lines
not meeting $Q_0$. Likewise, in Theorem~\ref{maintheorem} one may assume that
$\ell\cap Q_0=\emptyset$ (in which case $\ell\cap E=\emptyset$ automatically
and one does not need to condition on this.)
\end{lemma}

\begin{proof}
We present the argument for Theorem~\ref{alllines}; for Theorem
\ref{maintheorem} it is analogous.

Assume that \eqref{growthsizecovering} holds for all lines not meeting $Q_0$
and fix a sequence $a_k$ with $\frac{a_k}k\rightarrow\infty$. Given a dyadic
square
$Q\in\QQ_n$, let $A_Q$ be the event ``for every line $\ell$ not meeting $Q$,
the visible part $V_\ell(E\cap Q)$ can be covered by $4^{-n}a_k2^k$ dyadic
squares of side-length $2^{-k}$, for all large enough $k$''. By our assumption
for the sequence $4^{-n}a_k$ and the self-similarity of $E$, each $A_Q$ has
full probability, and so does the event
\[
A = \bigcap_{n=1}^\infty \bigcap_{Q\in\mathcal{Q}_n} A_Q.
\]

On the other hand, a line $\ell$ does not meet $E$ if and only if there is $n$
such that $\ell$ does not meet any square in $\mathcal{C}_n$. Clearly, if
$\ell$ is such a line, then
\[
V_\ell(E) \subset \bigcup_{Q\in\CC_n} V_\ell(Q\cap E).
\]
This inclusion shows that \eqref{growthsizecovering} holds whenever $A$ holds,
and thus it is an almost sure event.

The assertions on the Hausdorff and box dimensions follow easily from
\eqref{growthsizecovering}; see the proof of Theorem~\ref{alllines}.
\end{proof}

In the light of the previous lemma, we may assume that the line $\ell$ does not
meet $Q_0$. Horizontal and vertical lines are exceptional, and are easier to
handle; see \cite{J} for the proof of Theorem~\ref{maintheorem} in this case
(a slightly weaker version of Theorem~\ref{alllines} is also proved there; the
full version follows using the large deviation ideas used in this article).
Therefore from now on we will focus on the transversal case. We assume that
$\ell$ is of the form $y=-tx-a$, where $t,a>0$, since the other cases follow
by symmetry. Such a line will be fixed for the rest of this section.

Given $0<\eps<\frac 12$, we associate a set $Q(\eps)$ to each square $Q$ of
side
length $a$ as follows: $Q(\eps)$ is obtained by removing from $Q$ the half-open
squares of side length $\eps a$ from the upper left and the lower right
corners, see Figure~\ref{fig-block}. (For lines of positive slope, one would
need to remove the lower left and the upper right corners.)

The following theorem from \cite{RS} will play a crucial role in our study.
Recall that $Q_0$ denotes the closed unit square.

\begin{theorem}\label{RamsSimon}
Let $D\subset S^1$ be a closed connected arc such that
\[
D\cap\{(\pm 1,0),(0,\pm 1)\}=\emptyset.
\]
Then for any $0<\varepsilon<\frac 12$ there exists $q_\varepsilon>0$ such that
\[
\PP(\Pi_\ell(E)\supset\Pi_\ell(Q_0(\eps))\text{ for all }\ell\text{ with }
    \theta(\ell)\in D)=q_\varepsilon.
\]
Here $\theta(\ell)$ is the angle between $\ell^\perp$ and the $x$-axis.
\end{theorem}

\begin{proof}
Thhis is proved in \cite{RS}. For the convenience of the reader, a proof is also
sketched in the proof of Lemma~\ref{lem:mi2}.
\end{proof}

Given $Q\in\mathcal Q_n$, where $n\ge 3$, let
$\widetilde Q\in\mathcal Q_{n-2}$ be the unique dyadic square which contains
$Q$. We say that a square $Q$ is a \textbf{corner} if the relative position
of $Q$ within $\widetilde{Q}$ is either the upper left corner or
the lower right one.

Let $0<\varepsilon<\frac 13$ and let $n\ge 3$ be an integer. Denote the centre
of a square $Q$ by $z(Q)$. Given an interval $I\subset\Pi_\ell(Q_0)$ of length
$\varepsilon 2^{-n}$, we consider the collections
\[
\mathcal Q_I=\{Q\in\mathcal Q_n :\Pi_\ell(z(Q))\in I \}
\]
and
\[
\mathcal C_I = \mathcal Q_I \cap \mathcal C_n.
\]
The interval $I$ will be fixed for the moment. Write
\[
\mathcal{Q}_I = \{ Q_1,\ldots, Q_M\},
\]
where $\dist(z(Q_i),\ell) < \dist(z(Q_{i+1}),\ell)$ for
$i=1,\ldots,M-1$. Here $\dist(x,A)=\inf\{\vert x-a\vert:a\in A\}$ is
the distance between a point $x$ and a set $A$. Likewise, set
\[
\mathcal{C}_I = \{ C_1,\ldots, C_N\},
\]
where $\dist(z(C_i),\ell) < \dist(z(C_{i+1}),\ell)$ for
$i=1,\ldots, N-1$. Both $C_i$ and $N$ are random variables, while $Q_i$ and
$M$ are deterministic, but depend on the interval $I$.

Let $Z_i$ be the indicator function for the event ``$C_i$ is a
corner'' with the interpretation that $Z_i=0$ if $i>N$. Define
\[
X_m=\sum_{i=1}^m Z_i.
\]
Furthermore, let $\XX_m$ be the algebra generated by $X_1,\ldots, X_m$(or by
$Z_1,\ldots, Z_m$). The following technical lemma will be a crucial tool in the
proofs. It asserts that, whatever the distribution of corners and non-corners
among $C_1,\ldots C_{m-1}$ is, there is a uniformly positive probability that
the next chosen square $C_m$ (if defined) is not a corner.

\begin{lemma}\label{rarecorners}
There exists $\zeta<1$ depending only on $p$ (and not on $n$, $m$ or the
interval $I$) such that
\begin{equation}\label{rarecornersclaim}
\PP(Z_m=1 \mid \XX_{m-1})\le\zeta.
\end{equation}
\end{lemma}

We start by establishing three claims that will be useful in the
proof of the lemma.

\smallskip

\noindent\textbf{Claim 1}. For any $i\in \{1,\ldots,M-2\}$, at least one of
the successive squares $Q_i, Q_{i+1}, Q_{i+2}\in\mathcal Q_I$ is not a corner.

\smallskip

\begin{proof}[Proof of Claim 1]
\begin{figure}
 \centering
  \includegraphics[width=0.85\textwidth]{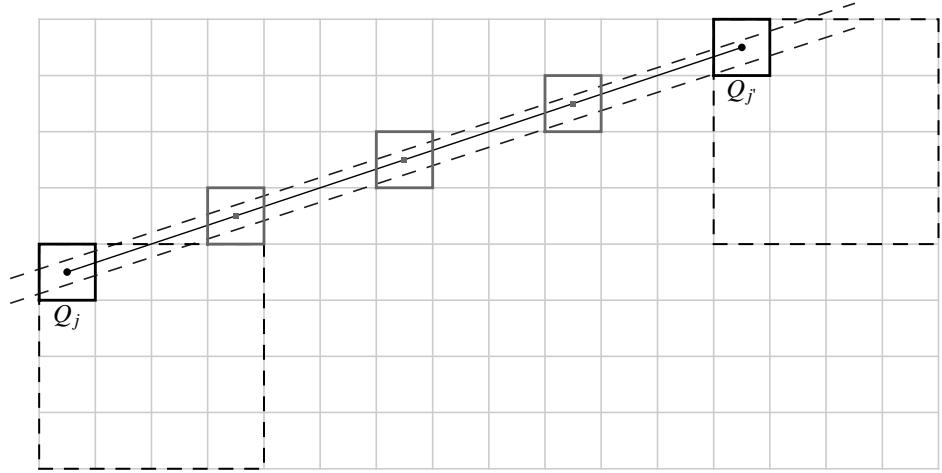}
\caption{The proof of Claim 1: the solid segment joining the centres of $Q_j$
and $Q_{j'}$ is $J$, and the parallel dashed lines represent the boundary of
the stripe $S$. A square is in $\mathcal{Q}_I$ if its centre lies on this
stripe. If $Q_j, Q_{j'}$ are in $\mathcal{Q}_I$ and are both corners of the
same type, then we can find three other squares between them with centres in
$J$, which are therefore also in $\mathcal{Q}_I$.} \label{fig-claim1}
\end{figure}
Suppose that $Q_i, Q_{i+1}, Q_{i+2}$ are all corners. Then there are
$j<j'\in\{i,i+1,i+2\}$ such that $Q_j$ and $Q_{j'}$ are corners of the same
type, i.e. both of them are either upper left or lower right corners. By
definition of $\mathcal{Q}_I$, $z(Q_j)$ and $z(Q_{j'})$ both lie in the stripe
$S$ of lines through $I$ orthogonal to $\ell$; see Figure~\ref{fig-claim1}.
Let $J$ denote the segment that joins $z(Q_j)$ and $z(Q_{j'})$, and denote its
length by $|J|$. By elementary algebra, the points on $J$ at distance
$\frac14|J|$, $\frac12|J|$ and $\frac34|J|$ from $z(Q_j)$ are all centres of
squares in $\mathcal{Q}_n$. Since $J$ is contained in $S$, this implies that
these three squares are in fact in $\mathcal{Q}_I$. Hence $j'-j\ge 4$, which
is a contradiction since we had assumed that $j'-j\in\{1,2\}$.
\end{proof}

\noindent\textbf{Claim 2}. Let $Q$, $\widehat Q\in\mathcal Q_I$ be
successive squares with $\dist(z(Q),\ell)<\dist(z(\widehat Q),\ell)$.
Then
\[
\PP(\widehat Q=C_1)\ge (1-p)\PP(Q=C_1).
\]

\begin{proof}[Proof of claim 2]
Let $R$ be the smallest dyadic square containing both $Q$ and $\widehat Q$, and
let $R_Q$ and $R_{\widehat Q}$ be the largest dyadic proper
subsquares of $R$ containing $Q$ and $\widehat Q$, respectively. Then
$R_Q\ne R_{\widehat Q}$. Denote by $A$ the event
``$R$ is chosen and there are no chosen squares in $\mathcal Q_I$ which
are closer to $\ell$ than those inside $R$''. As ``$Q=C_1$'' and
``$\widehat Q=C_1$'' are subevents of $A$, it is enough to prove
that
\[
\PP(\widehat Q=C_1\mid A)\ge(1-p)\PP(Q=C_1\mid A).
\]
Since $Q=C_1$ in particular implies that $Q$ is chosen, we have
\[
\PP(Q=C_1\mid A)\le\PP(Q\in\mathcal C_I\mid A)
   =\PP(\widehat Q\in\mathcal C_I\mid A).
\]
Conditioned on $A$,  the event ``$\widehat Q\in\mathcal C_I$ and
$R_Q$ is not chosen'' is a subevent of ``$\widehat Q=C_1$'', and moreover, the
events ``$\widehat Q\in\mathcal C_I$'' and ``$R_Q$ is not chosen'' are
independent conditioned on $R$ being chosen. This implies
\begin{align*}
\PP(\widehat Q=C_1\mid A)&\ge\PP(\widehat Q\in\mathcal C_I\text{ and }
R_Q \text{ is not chosen}\mid A)\\
&=\PP(\widehat Q\in\mathcal C_I\mid A)\PP(R_Q \text{ is not chosen}\mid A)\\
&\ge(1-p)\PP(Q=C_1\mid A).
\end{align*}
\end{proof}

\noindent \textbf{Claim 3}. Suppose that at least one square in $\QQ_n$ is not
a corner. Then
\begin{equation}\label{middlestep}
\PP(Z_1=0\mid \mathcal C_I\ne\emptyset) = \Omega(1).
\end{equation}

\begin{proof}[Proof of claim 3]
Denote the collection of corners by $\Cor$. We may write
$\Cor=\Cor_1\cup\Cor_2\cup\Cor_3$ where, for $i\le M-2$, the square
$Q_i\in\Cor_1$ if $Q_{i+1}\notin\Cor$ and $Q_i\in\Cor_2$ provided that
$Q_{i+1}\in\Cor$,
and $\Cor_3=\Cor\cap\{Q_{M-1},Q_M\}$.

According to Claim 1, for $j=1,2$
we may attach to any square $Q_i\in\Cor_j$ the square $Q_{i+j}\notin\Cor$.
Thus for any $Q_i\in\Cor_j$ ($j=1,2$) the events ``$Q_i=C_1$'' and
``$Q_{i+j}=C_1$'' are subevents of ``$\CC_I\ne\emptyset$ and
$C_1\notin \Cor_3$''. Write $A$ for the latter event. By Claim 2 we obtain that
\[
\PP(Q_{i+j}=C_1\mid A)\ge(1-p)^j
\PP(Q_i=C_1\mid A).
\]
Hence, using that $C_1\notin \Cor_3$,
\begin{align*}
1=&\left(\sum_{Q\in\Cor_1}+\sum_{Q\in\Cor_2}+\sum_{Q\notin\Cor}\right)
    \PP(Q=C_1\mid A)\\
  \le &\left(\frac 1{(1-p)^2}+\frac 1{1-p}+1\right)\sum_{Q\notin\Cor}
   \PP(Q=C_1\mid A),
\end{align*}
implying that $\PP(Z_1=0\mid A) = \Omega(1)$.

Since every $Q\in\mathcal Q_I$ has the same probability of being chosen, we
have
$\PP(Q_1=C_1)\ge\PP(Q_i=C_1)$ for all $i=2,\dots,M$, giving
$\PP(C_I\ne\emptyset)\le 3\PP(A)$. Hence
\begin{align*}
\PP(Z_1=0\mid C_I\ne\emptyset) &\ge\PP(Z_1=0 \textrm{ and } C_1\notin \Cor_3
  \mid  C_I\ne\emptyset)\\
 &\ge\frac13\PP(Z_1=0\mid A) = \Omega(1).
\end{align*}
This gives \eqref{middlestep}.
\end{proof}

Now we are ready to prove Lemma~\ref{rarecorners}.

\begin{proof}[Proof of Lemma~\ref{rarecorners}]
Let $\YY_m$ be the algebra generated by the random variables
$C_1,\ldots, C_{m\wedge N}$ and the event ``$m\le N$''. Note that this is a
refinement of $\XX_m$. Hence it is enough to prove that
\begin{equation}\label{rarecornersclaim2}
\PP(Z_m=1 \mid \YY_{m-1}) = 1 - \Omega(1).
\end{equation}

We assume $m\le N$; otherwise there is nothing to prove. Let $i_0$ be the index
for which $C_{m-1}=Q_{i_0}$. Note that $i_0<M$, since otherwise $m-1=N$.

\begin{figure}
  \centering
    \includegraphics[width=0.7\textwidth]{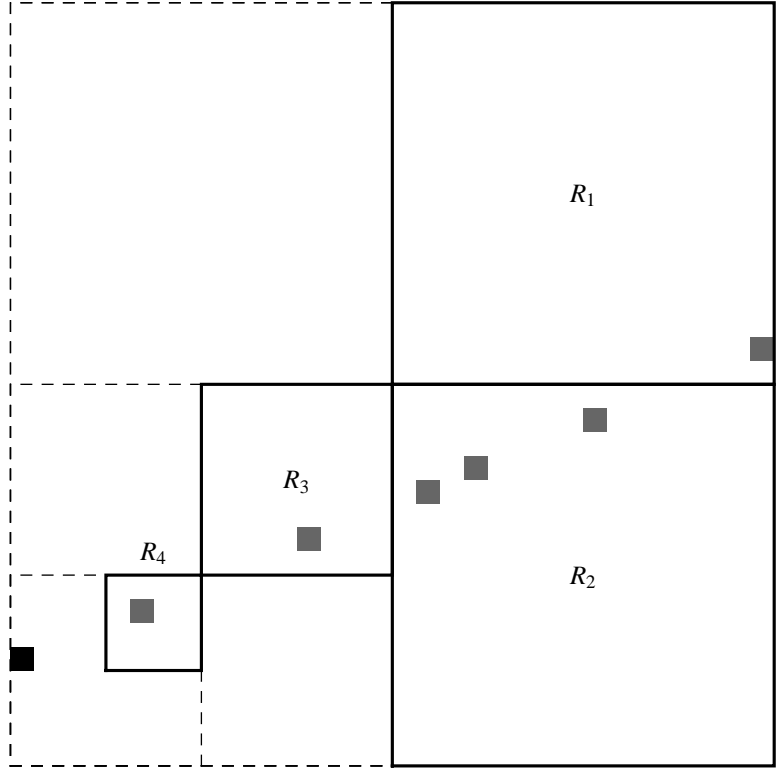}
\caption{Construction of the rectangles $R_i$: the black square represents
$C_{m-1}$, and the gray squares are the remaining squares in $\mathcal{Q}_I$
after $C_{m-1}$. } \label{fig-construction}
  \end{figure}

We select a finite collection $\{ R_i\}$ of dyadic squares inductively
in the following manner:  Let $R_1$ be the largest dyadic square which
contains $Q_M$ but does not contain $C_{m-1}$. Assuming that dyadic squares
$R_1,\dots, R_i$ have been selected, pick the largest index $i_0<j<M$ such that
$Q_j$ is not contained in $R_1\cup\dots\cup R_i$.
Let $R_{i+1}$ be the largest dyadic square which contains $Q_j$ but does
not contain $C_{m-1}$. The process stops when we have a collection
$\{R_1,\dots,R_L\}$ such that for all $i_0<j\le M$ the square $Q_j$ belongs
to $R_i$ for some unique $i=1,\dots,L$. See Figure~\ref{fig-construction}.

By construction, $C_{m-1}$ belongs to the dyadic square containing $R_i$ and
having side length twice of that of $R_i$ (see Figure~\ref{fig-construction};
these squares are represented by dotted lines). Therefore, the side length of
$R_{i+1}$ is at most that of $R_i$ for all $i=1,\dots,L-1$, and each $R_i$ has
probability $p$ of being chosen, independently of each other.

Assume first that all the squares after $C_{m-1}$ in $\mathcal Q_I$  are
corners. Then, by Claim 1, there are at most two of them, which gives $L\le 2$.
Thus the probability that neither of the two corners in $\mathcal Q_I$ after
$\mathcal C_{m-1}$ is chosen is
at least $(1-p)^2$, giving
\[
\PP(Z_m=1\mid \YY_{m-1})\le 1-(1-p)^2.
\]

Now assume that there is $R_i$ containing at least one square in $\QQ_I$ which
is not in $\Cor$.  To see that \eqref{rarecornersclaim2} holds, divide the
collection $\{R_1,\dots,R_L\}$ into two parts $P_{\textrm{bad}}$ and
$P_{\textrm{good}}$ as follows: we say that $R_i\in P_{\textrm{bad}}$ if all
squares that belong to $\mathcal Q_I$ and are contained in $R_i$ are corners,
and $R_i\in P_{\textrm{good}}$ if $R_i$ contains
a square that belongs to $\mathcal Q_I$ and is not a corner.

Since each $R_i$ contains some square in $\mathcal Q_I$, we may use Claim 1 as
in the proof of Claim 3 to find that we may attach to any
$R_i\in P_{\textrm{bad}}$, with $i\le L-2$, a square
$R_{i+j}\in P_{\textrm{good}}$ where $j=1$ or $2$. The same argument of Claim 3
then gives
\[
\PP(C_m\subset R_i \textrm{ for some } R_i\in P_{\textrm{good}} \mid \YY_{m-1})
  = \Omega(1).
\]
(Recall that we are conditioning on $P_{\textrm{good}}$ being non-empty.) Hence
it remains to prove that
\[
\PP(Z_m=0|C_m\subset R_i \textrm{ for some } R_i\in P_{\textrm{good}}, \YY_{m-1})
   = \Omega(1).
\]
However, by conditioning on the index $i$ for which $C_m\subset R_i$, we are
exactly in the situation of Claim 3 (applied to some $n'<n$ and a different
interval $I'$).

This completes the proof of the lemma.
\end{proof}

As a corollary, we obtain the following large deviation bound for $X_m$:

\begin{lemma}[Azuma-Hoeffding inequality]\label{azuma}
Let $\zeta$ be as in Lemma~\ref{rarecorners} and choose $\eta>0$ such that
$\zeta+\eta<1$. Then
\[
\PP(X_m>(\zeta+\eta)m)<e^{-\frac{\eta^2 m}2}.
\]
\end{lemma}

\begin{proof} Define $Y_i=Z_i-\zeta$ and $\widetilde X_m=\sum_{i=1}^m Y_i$.
Then $\widetilde X_m$ is a (discrete time) supermartingale, that is,
$\EE(\widetilde X_m\mid\widetilde X_1,\dots,\widetilde X_{m-1})
   \le\widetilde X_{m-1}$. Applying the
Azuma-Hoeffding inequality \cite[Theorem 7.2.1]{ASE} to $\widetilde X_m$
with $\lambda=\eta\sqrt m$ gives the claim. Note that \cite[Theorem 7.2.1]{ASE}
is verified only for martingales but the same proof works for
supermartingales as well.
\end{proof}

\section{Visible parts from lines}\label{linessection}

This section is dedicated to the proofs of Theorems~\ref{alllines} and
\ref{maintheorem}, and Corollary~\ref{aboveperco}. We start with
Theorem~\ref{maintheorem} for clarity of exposition, as the proof is somewhat
easier than that of Theorem~\ref{alllines}.

\begin{proof}[Proof of Theorem~\ref{maintheorem}] As remarked in the previous
section, it is enough to prove the theorem for a fixed line $\ell = -tx-a$
with $t,a>0$. By Theorem~\ref{RamsSimon},
$\HH^1(V_\ell(E))>0$ almost surely
conditioned on non-extinction, and therefore we only need to prove that
\[
\HH^1(V_\ell(E))<\infty
\]
almost surely.

Denote by $\theta$ the angle between $\ell^\perp$ and the positive
$x$-axis, and let $\varepsilon<\tfrac12\sin\theta\cos\theta$. (The factor
$\sin\theta$ is needed when $\theta$ is close to 0 and the factor $\cos\theta$
is essential when $\theta$ is close to $\frac\pi 2$.)
Given a positive integer $n$, let $N(n)$ be the smallest integer such that
$N(n)\varepsilon2^{-n}\ge\sqrt 2$. Then $N(n)\le 2\eps^{-1} 2^n$.
Divide $\Pi_\ell(Q_0)$ into disjoint line segments of length
$\varepsilon2^{-n}$ (except for the last one which may be smaller), and denote
them by $I_{n,1},\ldots, I_{n,N(n)}$.
For all $1\le j\le N(n)$, set $\mathcal Q_{n,j}=\mathcal Q_{I_{n,j}}$.

We say that $Q\in\mathcal Q_{n,j}$ \textbf{induces a block} if $Q$ is not a
corner and the unique square $\widetilde Q\in\mathcal Q_{n-2}$ which contains
$Q$ is a \textbf{block}, meaning that
\[
\Pi_\ell(\widetilde Q(\frac18))\subset\Pi_\ell(\widetilde Q\cap E).
\]
If $Q$ is not a corner and $\widetilde Q$ is not a block, we say that
$\widetilde Q$ is a \textbf{window} and $Q$ \textbf{induces a window}.
By Theorem~\ref{RamsSimon} and independence, every
chosen square $Q\in\mathcal Q_{n,j}$ which is not a corner has the same
probability $q>0$ of inducing a block. Moreover,
if $\widetilde Q_1$ and $\widetilde Q_2$ are chosen and different, then the
events ``$\widetilde Q_1$ is a block'' and ``$\widetilde Q_2$ is a block'' are
independent.

\begin{figure}
  \centering
   \includegraphics[width=0.7\textwidth]{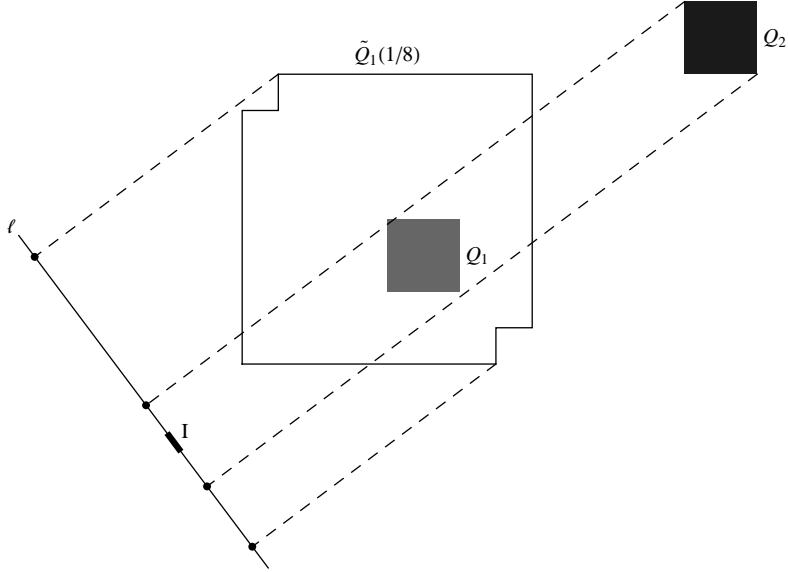}
\caption{In this figure $Q_1,Q_2\in\mathcal{C}_I$. The number $\varepsilon$
(i.e. the length of $I$ relative to the side length of $Q_1$ and $Q_2$) is
chosen so that the projection of $Q_2$ onto $\ell$ is contained in the
projection of $\widetilde{Q}_1(\frac 18)$ whenever $Q_1$ is not a corner. When
$Q_1$ induces a block, the visible part of $E$ from the interval $I$ cannot
intersect $Q_2$.} \label{fig-block}
\end{figure}

The geometric significance of blocks is depicted in Figure~\ref{fig-block}: Let
$Q_1$, $Q_2\in\mathcal Q_{n,j}$ be squares such that $Q_1$ is closer to $\ell$
than $Q_2$ and
$\widetilde Q_1\ne\widetilde Q_2$. Suppose that $Q_1$ induces a block. Then
by the choice of $\varepsilon$ we have
\[
\Pi_\ell(Q_2)\subset \Pi_\ell(\widetilde Q_1(\frac18))
\subset\Pi_\ell(\widetilde Q_1\cap E),
\]
giving $Q_2\cap V_\ell(E)=\emptyset$. In particular, if
$Q_B\in\mathcal Q_{n,j}$ is the first square in $\QQ_{n,j}$ that induces a
block, then we can cover the visible part of $E$ from $I_{n,j}$ by all chosen
squares in $Q_{n,j}$ up to $Q_B$, plus the squares $Q$ such that
$\widetilde{Q} = \widetilde{Q}_B$. Thus, estimates on the position of the first
square in $\QQ_{n,j}$ that induces a block will yield estimates on the size of
$V_\ell(E)$.

Letting $\zeta$ and $\eta$ be as in Lemma~\ref{azuma}, define
$\gamma=1-(\zeta+\eta)$ and $\gamma'=e^{-\eta^2/2}$. Denote by $Y_{n,j}$
the number of chosen squares in $\mathcal Q_{n,j}$ which are
needed to cover the stripe of $V_\ell(E)$ above $I_{n,j}$, and assume that
$Y_{n,j}=i+4$. Now there are two possibilities: the number of corners among
the first $i$ chosen squares in $\mathcal Q_{n,j}$ is either at least
$(\zeta+\eta)i$ or less than $(\zeta+\eta)i$.

By Lemma~\ref{azuma}, the first event has probability at most $\gamma'^i$ of
occurring. In the latter case
the number of squares that induce a window among the first $i$ squares is at
least $\gamma i$. Observe also that for given $Q\in\mathcal Q_{n,j}$ there are
at most four $Q'\in\mathcal Q_{n,j}$ (including $Q$) such that
$\widetilde Q'=\widetilde Q$. Hence the probability of the second event is at
most $(1-q)^{\frac{\gamma i}4}$. We deduce that
\[
\PP(Y_{n,j}=i+4)\le (1-q)^{\frac{\gamma i}4}+\gamma'^i\le 2\tilde\gamma^i
\]
where $\tilde\gamma=\max\{(1-q)^{\frac{\gamma}4},\gamma'\}<1$. This in turn
implies that $\EE(Y_{n,j})= O(1)$. Writing $S_n=\sum_{j=1}^{N(n)} Y_{n,j}$, we
therefore have
\begin{equation}\label{expectednumber}
\EE(S_n) = O(\eps^{-1} 2^n).
\end{equation}
By definition, we can cover $V_\ell(E)$ by $S_n$ squares of side length
$2^{-n}$, whence
\[
\HH^1(V_\ell(E)) \le \liminf_{n\rightarrow\infty}\sqrt2\cdot 2^{-n} S_n.
\]
By Fatou's lemma, Lemma~\ref{measurability} below and inequality
\eqref{expectednumber}, we have that almost surely
\[
\EE(\HH^1(V_\ell(E)))\le\liminf_{n\rightarrow\infty}\sqrt2 \cdot 2^{-n}\EE(S_n)
   =O(\eps^{-1})<\infty.
\]
This shows that $\HH^1(V_\ell(E))<\infty$ almost surely, as desired.
\end{proof}

\begin{proof}[Proof of Corollary~\ref{aboveperco}]
The claim follows from Theorem~\ref{maintheorem} combined with
Fubini's theorem.
For the purpose of applying Fubini's theorem we need to prove that the set
$\{(E,\ell) : 0<\mathcal H^1(V_\ell(E))<\infty\}$ is measurable.
This is an immediate consequence of Lemma~\ref{measurability} in which
we prove that it contains a Borel set with full measure.
\end{proof}

Let $S_n=S_n(E,\ell)$ be as in the proof of Theorem~\ref{maintheorem}, that is,
$S_n(E,\ell)$ is the number of the dyadic squares of side length $2^{-n}$
that cover $V_\ell(E)$. In the proof of Theorem~\ref{maintheorem} we estimate
$S_n(E,\ell)$ from above by a function which is defined by counting blocks,
windows and corners. Call this function $\widetilde S_n(E,\ell)$.
In the space of constructions we use the
natural topology induced by the open cylinder sets
$[F]=\{E:E_m=\cup_{Q\in F}Q\}$ where
$F\subset\mathcal Q_m$ and $E_m$ is the union of all
chosen squares of side length $2^{-m}$ in the construction of $E$, that is,
$E=\cap_{m=1}^\infty E_m$.

\begin{lemma}\label{measurability}
The function $(E,\ell)\mapsto\widetilde S_n(E,\ell)$ is a Borel function for
all positive integers $n$.
\end{lemma}

\begin{proof}
Since the corners are independent of $E$ and $\ell$ we may consider only blocks
and windows. Let $N$ be a positive integer. The set
$\{(E,\ell):\widetilde S_n(E,\ell)\le N\}$ is a finite union of finite
intersections of sets of the form $\{(E,\ell): Q\text{ is a block}\}$
and $\{(E,\ell): Q\text{ is a window}\}$ where $Q\in\mathcal Q_{n-2}$. Since
the latter set is the complement of the former one it suffices to verify that
the former one is a Borel set.

From the definition of a block we get
\begin{eqnarray*}
\{(E,\ell) : Q \text{ is a block} \}&
=&\{(E,\ell) : \Pi_{\ell}(Q\cap E)\supset \Pi_{\ell}(Q(\frac{1}{8}))\} \\ &
=&\bigcap_{m=1}^{\infty}\{(E,\ell) : \Pi_{\ell}(Q\cap E_m)
 \supset \Pi_{\ell}(Q(\frac{1}{8})) \}
\end{eqnarray*}
where the last equality follows from the fact that if
$y \in \Pi_{\ell}(Q(\frac18))$ and
$\Pi_{\ell}(Q(\frac18))\subset\Pi_{\ell}(Q \cap E_m)$ for all $m$
then the sets $\Pi_{\ell}^{-1}(y) \cap Q \cap E_m$ form a decreasing sequence
of non-empty compact sets, and therefore, there exists
$x \in\Pi_{\ell}^{-1}(y) \cap E\cap Q$ giving $y\in\Pi_\ell(Q\cap E)$.

Given $m$, the set $\mathcal Q_m$ has a finite number of subsets, say
$F_1,\dots,F_M$. Now
\[
\{(E,\ell): \Pi_{\ell}(Q\cap E_m)\supset\Pi_{\ell}(Q(\frac18))\}
=\bigcup_{i=1}^M ([F_i]\times
  \{\ell :\Pi_\ell(Q\cap\bigcup_{Q'\in F_i}Q')\supset \Pi_{\ell}(Q(\frac18))\})
\]
is a Borel set since for fixed $i$ the set
$\{\ell :\Pi_\ell(Q\cap\bigcup_{Q'\in F_i}Q')\supset \Pi_{\ell}(Q(\frac18))\}$
consists of finitely many closed intervals. This finishes the proof.
\end{proof}

In the last part of this section we prove Theorem~\ref{alllines}.

\begin{proof}[Proof of Theorem~\ref{alllines}]
By Theorem~\ref{RamsSimon} (and the results of \cite{FG})
$\dimh V_\ell(E)\ge 1$ for all $\ell$ almost surely. Since
$\dimh A\le\ldimb A\le\udimb A$ for any bounded set $A$, it is enough to
show that, given a sequence $a_n$ with $\frac{a_n}n\rightarrow\infty$, almost
surely the following holds: if $\ell$ is a line not meeting $Q_0$, then
\[
N_n(V_\ell(E)) \le a_n 2^n \quad\textrm{for all large enough } n.
\]
Indeed, by Lemma~\ref{enoughnotmeetingsquare} it is enough to consider lines
which do not meet the unit square, and if the above holds then clearly
$\udimb V_\ell(E)\le 1$ (taking for example $a_n=n^2$).

Let $D$ be a closed interval of directions which does not contain the vertical
or horizontal ones. Recall that the direction of a line $\ell$ is parametrised
by the angle between $\ell^\perp$ and the $x$-axis and is denoted by
$\theta(\ell)$. It is enough to prove the claim for all lines with directions
in $D$ simultaneously, since we can cover all directions by a countable union
of such intervals plus the horizontal and vertical directions.
Observe that $V_\ell(E)=V_{\ell'}(E)$ if $\ell'$ is parallel to $\ell$ and
they both are on the same side of the unit square. By symmetry, $V_\ell(E)$ and
$V_{\ell'}(E)$ still have the same distribution if $\ell'$ and $\ell$ are
parallel but on different sides of the unit square.

Choose $\varepsilon>0$ such that
$\varepsilon<\frac 12\sin\theta\cos\theta$ for all $\theta\in D$. Consider
$n\in\mathbb N$ and a line $\ell$ with $\theta(\ell)\in D$. Let $I$ be a
line segment of length $\varepsilon 2^{-n}$ in $\Pi_\ell(Q_0)$. We say that a
square $Q$ is \textbf{above} $I$ if its centre projects inside $I$ under
$\Pi_\ell$. Such an interval $I$ is \textbf{good} if either there are fewer
than $a_n$ chosen squares above $I$, or if there is a chosen square among the
first $a_n$ chosen squares above $I$ which is not a corner and which induces a
block for \textit{all} $\theta\in D$. Intervals which are not good will be
called \textbf{bad}.

Suppose there are at least $a_n$ chosen squares above $I$. Letting $\zeta,\eta$
be as in Lemma~\ref{azuma} we may, as in the proof of
Theorem~\ref{maintheorem}, consider the cases in which the number of corners
among the first $a_n$ chosen squares is at least $(\zeta+\eta)a_n$ or less than
$(\zeta+\eta)a_n$. Arguing exactly like in the proof of
Theorem~\ref{maintheorem}, but using the full strength of
Theorem~\ref{RamsSimon} which holds simultaneously for all directions in $D$,
we obtain that, for any given interval $I$,
\[
\PP(I \textrm{ is bad}) < e^{-\Omega(a_n)}.
\]

Let $0<\varepsilon'<\varepsilon$. Divide $\Pi_\ell(Q_0)$ into line segments of
length $\varepsilon' 2^{-n}$ as in the proof of Theorem~\ref{maintheorem}.
Let $I_\ell'$ be such a line segment and let $I\supset I_\ell'$ be a line
segment of length $\varepsilon 2^{-n}$ having the same centre as $I_\ell'$.
Denote by $S_I$ the stripe generated by $I$, that is, $S_I=I\times\ell^\perp$
where $\ell$ is the line containing $I$. Choose $\delta>0$ so small that
$S_{I_{\ell_1}'}\cap Q_0\subset S_I\cap Q_0$ for all $\ell_1$ such that
\begin{equation}\label{close}
|\theta(\ell_1)-\theta(\ell)|<\delta 2^{-n},
\end{equation}
where $I_{\ell_1}'$ is the line
segment of length $\varepsilon' 2^{-n}$ in $\Pi_{\ell_1}(Q_0)$ which is
closest to $I_\ell'$. Observe that if $I$ is good then the visible part from
$I_{\ell_1}'$ is covered by the first $a_n$ chosen squares above $I$ for all
$\ell_1$ satisfying \eqref{close} (or by all such chosen squares if there are
fewer than $a_n$ of them).

Since for each $\ell$ we need to consider less than $2\varepsilon'^{-1}2^n$
intervals, the probability that there is at least one interval $I'$ such
that we cannot cover the visible part above $I'$ by at most $a_n$ squares
of side length $2^n$ is less than $2\varepsilon'^{-1}2^ne^{-\Omega(a_n)}$.
By the above observation, if we have this property for a set of lines
$\{\ell_i\}$ such that the set of directions $\{\theta(\ell_i)\}$ is
$(\delta 2^{-n})$-dense, then
it is true for all directions in $D$. Therefore, the
probability that there is some interval $I'\subset\Pi_{\ell'}(Q_0)$ for some
$\ell'$ with $\theta(\ell')\in D$ such that we need more than $a_n$ squares
to cover the visible part from $I'$, is bounded above by
\[
P_n:=4(\delta\varepsilon')^{-1}2^{2n}e^{-\Omega(a_n)}.
\]
By our assumption that $\frac{a_n}n\rightarrow\infty$, the series $\sum_n P_n$
converges. Hence the Borel-Cantelli lemma implies that almost surely for each
$\ell$ with $\theta(\ell)\in D$, the visible part $V_\ell(E)$ satisfies
\[
N_n(V_\ell(E))\le 2\varepsilon'^{-1}2^n a_n\quad\textrm{for all large enough }n.
\]
Replacing $a_n$ by $a'_n = a_n \varepsilon'/2$ we obtain the desired statement.
\end{proof}

\section{Visible parts from points}\label{pointssection}

In this section we consider visible parts from points. The same general ideas
apply, except that we need an analogue of Theorem~\ref{RamsSimon} for radial
projections. This is given by the following proposition. For
$x\in\mathbb R^2\setminus Q_0$, we denote by $\Pi_x$ the radial projection onto
a circle $S(x)$ centred at $x$ and not intersecting $Q_0$.

\begin{proposition}\label{bigarc}
Fix $x^0\in\mathbb R^2\setminus Q_0$ and let
$r_0=\frac 1{10}\min\{1,\dist(x^0,Q_0)\}$.
Then for any $0<\varepsilon<\frac 12$ there exists $q_\varepsilon>0$ such that
\[
\PP\bigl(\Pi_x(E)\supset\Pi_x(\underline{Q}_0(\eps))\text{ for all }x\in
  B(x^0,r_0)\bigr)=q_\varepsilon.
\]
Here $\underline{Q}_0(\varepsilon)$ is the set obtained by removing half-open
squares of side length $\varepsilon$ from each corner of the unit square.
\end{proposition}

The proof of this proposition will be given at the end of this section. We now
state the counterparts of Theorems~\ref{alllines} and \ref{maintheorem} for
visible parts from points.

\begin{theorem}\label{allpoints}
Let $p>\frac12$. Conditioned on non-extinction, almost surely
\[
\dimh V_x(E)=\dimb V_x(E)=1
\]
for all $x\in\mathbb R^2\setminus E$. Moreover, if $a_n$ is any sequence such
that $\frac{a_n}{n^2}\rightarrow\infty$ as $n\rightarrow\infty$, then almost
surely
\[
N_n(V_x(E)) \le a_n 2^n \quad\textrm{for all sufficiently large } n
\]
for all $x\in\mathbb R^2\setminus E$.
\end{theorem}

\begin{proof}
The counterpart of Lemma~\ref{enoughnotmeetingsquare} is valid also in this
case, so we may assume that $x\notin Q_0$.
The proof is similar to the proof of Theorem~\ref{alllines} for those
$x=(x_1,x_2)$ which satisfy $x_1\notin [0,1]$ and $x_2\notin [0,1]$. In this
case the direction of all the rays from $x$ to $Q_0$ is at a positive distance
from the horizontal/vertical ones. Then for a fixed small enough
$\varepsilon>0$ we can divide $\Pi_x(Q_0)$ into arcs of angular length
$(|x|+1)^{-1}\varepsilon 2^{-n}$, and then argue like in Theorem~\ref{alllines},
using Proposition~\ref{bigarc} instead of Theorem~\ref{RamsSimon}.

The remaining points induce horizontal or vertical rays. Let $x$ be such a
point. To deal with the singularity, we cover the arc $D=\pi_x(Q_0)$ by subarcs
$D_j$ of length $c 2^{-j}$, so that the distance from $D_j$ to the
vertical/horizontal line is comparable to $\varepsilon_j = 2^{-j}$.

Now fix a scale $2^{-n}$. The visible part from $x$ along rays in $D_j$ with
$j> n$ can be covered using all squares in $\mathcal{Q}_n$ intersecting such
rays; there are $O(2^n)$ such squares. For each fixed $j\le n$, we can argue
exactly as in the proof of Theorem~\ref{alllines} (using
Proposition~\ref{bigarc} instead of Theorem~\ref{RamsSimon}) to find that the
expected number of squares of side length $2^{-n}$ needed to cover the part of
$V_x(E)$ corresponding to $D_j$ is $O(2^{-j}\varepsilon_j^{-1} 2^n) = O(2^n)$.
Moreover, writing $b_n=\frac{a_n}n$, the probability that one needs more than
$b_n 2^n$ squares is at most $e^{-\Omega(b_n)}$. Therefore with probability
$1-ne^{-\Omega(b_n)}$ one can cover $V_x(E)$ by $n b_n 2^n=a_n 2^n$ squares in
$\mathcal{Q}_n$.

This argument is for a fixed point $x$, but similarly as in the proof of
Theorem~\ref{alllines}, a bound that works for $x$ works also in a
neighbourhood of $x$ (at the cost of losing a constant), and we can cover any
bounded part of $\mathbb R^2\setminus Q_0$ by exponentially many such
neighbourhoods. The proof then finishes in the same way as the proof of
Theorem~\ref{alllines}.
\end{proof}

\begin{theorem}\label{points}
Let $p>\frac12$. Write
\[
\mathcal{D} = \{ x=(x_1,x_2): x_1\notin [0,1] \textrm{ and } x_2\notin [0,1]\}.
\]
If $x\in\mathcal{D}$, then $V_x(E)$ has finite $\mathcal{H}^1$-measure almost
surely. For any $x\notin E$, the visible part $V_x(E)$ has $\sigma$-finite
$\mathcal{H}^1$-measure almost surely.
Furthermore, conditioned on non-extinction, almost surely
\[
0 < \mathcal{H}^1(V_x(E)) < \infty
\]
for $\mathcal L^2$-almost all $x\in\mathcal D$, and $V_x(E)$ has positive and
$\sigma$-finite $\mathcal H^1$-measure for $\mathcal L^2$-almost all
$x\in\mathbb{R}^2\setminus E$.
\end{theorem}

\begin{proof}
If $x\in\mathcal{D}$, then the proof is similar to the proof of
Theorem~\ref{maintheorem}, with the main modifications being the same ones as
in Theorem~\ref{allpoints}.

Now assume that $x_1\in [0,1]$ or $x_2\in [0,1]$. The value of $\varepsilon$
required becomes $0$ at the horizontal or vertical lines. Hence we consider
countably many subarcs covering all directions but horizontal/vertical. As
before, the Hausdorff measure of the visible part from each subarc is finite
almost surely, so we obtain that $V_x(E)$ has $\sigma$-finite measure almost
surely, as desired.

The latter assertion follows easily by Fubini's theorem.
\end{proof}

We finish the section with the proof of Proposition~\ref{bigarc}.

\begin{proof}[Proof of Proposition~\ref{bigarc}]
Let us begin with two remarks. First, it is enough to prove this proposition
for some fixed value of $\varepsilon=\varepsilon_0$. Indeed, it will
immediately imply the assertion for any $\varepsilon>\varepsilon_0$. On the
other hand, with positive probability all the four first level subsquares
belong to $\mathcal C_1$. Therefore if we know the assertion is satisfied for
$\varepsilon_0$ for each of them with positive probability, we obtain the
assertion for $\frac{\varepsilon_0}2$. (To see this, it is useful to note that
for any $\varepsilon<\frac 12$, $\underline{Q}_0(\varepsilon)$ contains a
``plus sign'' formed by lines parallel to the sides bisecting the square in
two equal parts. Moreover, the union of the projections of the plus signs in
each square in $\mathcal Q_1$ contains the projection of the plus sign in
$Q_0$.)

The second remark is that we can freely assume that $x^0$ is arbitrarily far
away from $Q_0$. Indeed, again with positive probability all the four first
level subsquares belong to $\mathcal C_1$ and the (relative) distance from
$x^0$ to each of them is already at least two times greater than the (relative)
distance from $x^0$ to $Q_0$. Repeating this, we only need to know the
assertion for $x^0$ at very large distance from $Q_0$ to prove the assertion for
all $x^0\in \mathbb R^2 \setminus Q_0$.

There will be two cases: $x^0$ is in a direction approximately
horizontal/vertical
from $Q_0$, or $x^0$ lies in a ``diagonal'' direction. For notational
simplicity we translate the picture so that $Q_0$ is centred at the origin. By
symmetry, it is enough to consider the cases stated in Lemmas~\ref{lem:mi1} and
\ref{lem:mi2} below, which completes the proof.
\end{proof}

\begin{lemma} \label{lem:mi1}
The assertion of Proposition~\ref{bigarc} is satisfied for
$\varepsilon=\frac 14$ and $x^0=(x_1, x_2)$ such that $x_2<0$, $x_1<-N_1$ and
$\frac{x_1}{x_2}> N_1$ for $N_1$ large enough.
\end{lemma}

\begin{proof}
Let us introduce some notation. We will call a line $\ell$ \textbf{passing}
through a square $Q$ if it intersects two parallel sides of $Q$. Note that,
provided $N_1$ is sufficiently large, any line containing $x\in B(x^0,r_0)$ and
intersecting $Q_0(\frac 14)$ is passing through one of the sixteen second
level subsquares of $Q_0$ (hitting their vertical sides). As each of
those subsquares has positive probability of belonging to $\mathcal C_2$, it is
enough to prove that with positive probability all the lines containing
$y\in B(y^0,r_0)$
and passing through $Q_0$ intersect $E$, where $y^0=(y_1, y_2)$ satisfies
$y_2<0$, $y_1<-4N_1$ and $\frac{y_1}{y_2}>\frac{N_1}2$.

Given $k\in \mathbb N$ and $z\in \Pi_y(Q_0)$, let $V_k(y,z)$ be the
number of squares $Q\in \mathcal C_k$ passed by the line $\ell(y,z)$ going
through $y$ and $z$. We denote by $Z_y$ the subarc of $\Pi_y(Q_0)$
determined by the lines $\ell(y,z)$ passing through $Q_0$.

Let $n$ be so large that
\begin{equation} \label{eqn:inc}
(2^n-1)p^n > 2
\end{equation}
and let $N_1=2^{n+1}$. This is the point where we use that $p>\frac 12$. As is
easy
to check, every line containing $y$ and
passing through $Q_0$ intersects at most $2^n+1$ of the $n^{\text{th}}$ level
subsquares of
$Q_0$, passing through at least $2^n-1$ of them. Hence, by \eqref{eqn:inc},
for each of those lines the expected number of squares in $\mathcal C_n$
passed by the line is greater than 2.

We want to apply an appropriate large deviation theorem to show that with
positive probability, for each $y$ and $z$ the function $V_k(y,z)$ will
actually increase exponentially fast with $k$. This
will in particular imply that $\ell(y,z)$ has non-empty intersection with
$\bigcup_{\mathcal C_k} Q$ for all $k$, thus with $E$ as well, which is
precisely the statement we need.

We parametrise the space of lines $L=\{\ell(y,z):y\in B(y^0,r_0), z\in Z_y\}$
by their intersection point with the vertical line $x_1=-5N_1$ and by the
angle they make with the $x$-axis. We call this parameter set $P$. (The
particular parametrisation chosen is not important.)

Denote by $\{ w_i^{(k)}\}$ the set of corner points of all
subsquares of $Q_0$ of level $k$. For each $i$ the condition
$w_i^{(k)}\in\ell(y,z)$ defines a smooth curve $\gamma_i$ on $P$. These
curves divide $P$ into components denoted by $\{C_j^{(k)}\}$.
Each $C_j^{(k)}$ is such that for
any two lines $\ell_1, \ell_2\in C_j^{(k)}$ the set of subsquares of $Q_0$ of
level $k$ passed by $\ell_1$ and by $\ell_2$ is the same (and the boundary
lines of
each $C_j^{(k)}$ pass through the same subsquares the other lines in
$C_j^{(k)}$ pass through, plus possibly some additional ones). Hence,
$V_{k}(y,z)$ is constant on each $C_j^{(k)}$ (and can only increase at the
boundary points).

We claim that the number of these components is at most
$2^{4 k}$. Note that the components are faces of the planar graph whose
vertices are the intersection points of the curves $\gamma_i$ and edges are
the pieces of $\gamma_i$ between vertices. By Euler's theorem, the number of
faces is less than twice the number of vertices. Since there is at most
one line going through $w_i^{(k)}$ and $w_j^{(k)}$ for $i\ne j$, $\gamma_i$
and $\gamma_j$ intersect at most once. Thus the number of vertices is at most
$\frac{N^2}2$, where $N=(2^{k}+1)^2$ is the number of corner points. This
yields our claim.

For each $k\ge 1$, let $\{ \ell(y_j^{(kn)},z_j^{(kn)}) \}$ be a collection of
representatives of the components $\{ C_j^{(kn)}\}$. Let $A_{k,j}$ be the event
\[
V_{kn}(y_j^{(kn)},z_j^{(kn)}) \ge 2^{k}.
\]
Further, let $A_k = \bigcap_j A_{k,j}$. Because of the way the components
$C_j^{(kn)}$ were defined, it will be enough to show that
$\PP(\cap_{k=1}^\infty A_k) = \Omega(1)$.

There is a positive probability that $V_n(y,z)\ge 2$ for all $\ell(y,z)\in P$.
Indeed, it is enough that all squares of generation $n$ are chosen. Thus,
$p_0:=\PP(A_1) >0$.

Now suppose that $A_k$ holds, and consider a line
\[
\ell_j = \ell(y_j^{((k+1)n)},z_j^{((k+1)n)}).
\]
By assumption, $\ell_j$ passes through at least $2^k$ squares in
$\mathcal{C}_{kn}$. By \eqref{eqn:inc}, if $Q$ is one of these squares, the
expected number of squares in $\mathcal{C}_{(k+1)n}$ that $\ell_j$ hits inside
$Q$ is strictly greater than $2$. Thus, conditioned on $\mathcal{C}_{kn}$,
$V(y_j^{((k+1)n)},z_j^{((k+1)n)})$ is the sum of at least $2^k$ i.i.d. bounded
random variables with expectation $E>2$. Note that the distribution of these
random variables is independent of $k$. By standard large deviation results
(for example one could use the Azuma-Hoeffding inequality
\cite[Theorem 7.2.1]{ASE} as
in the proof of Lemma~\ref{azuma}), we see that
\[
\PP(V(y_j^{((k+1)n)},z_j^{((k+1)n)})\ge 2^{k+1}) \ge 1-\gamma^{2^k},
\]
for some $\gamma<1$ which does not depend on $k$ or $j$. In other words,
$\PP(A_{k+1,j})\ge 1-\gamma^{2^k}$.

The events $A_{k,j}$ are clearly increasing, whence we can apply the
FKG-inequality \cite[Theorem 2.4]{G} to obtain
\[
\PP(A_k) \ge \prod_j \PP(A_{k,j}) \ge (1-\gamma^{2^k})^{2^{4(k+1)n}}.
\]
Therefore
\begin{align*}
\PP\left(\bigcap_{k=1}^\infty A_k\right) &= \PP(A_1) \prod_{k=1}^\infty
  \PP(A_{k+1}|A_k)\\
&\ge p_0 \prod_{k=1}^\infty \left(1-\gamma^{2^k}\right)^{2^{4(k+1)n}}.
\end{align*}
Since $\gamma^{2^k}$ goes to $0$ superexponentially fast while $2^{4(k+1)n}$
grows only exponentially fast, the infinite product converges. This completes
the proof.
\end{proof}

The second case was essentially done in \cite{RS} and the proof is very similar
to the proof of Lemma~\ref{lem:mi1}, but for completeness we will remind here
the basic steps of the proof. At the same time, since the proof is very
similar, we give a sketch of the proof of Theorem~\ref{RamsSimon}.

\begin{lemma} \label{lem:mi2}
There exists $N_2>0$ such that if $x^0=(x_1, x_2)$ satisfies $x_1<0$, $x_2<0$,
$1\le x_1/x_2<N_1$ and $x_1+x_2<-N_2$ then the assertion of Proposition
\ref{bigarc} is satisfied for $x^0$ with $\varepsilon = \frac 14$.
\end{lemma}

\begin{proof}
We begin with some notation. Given $Q$, a subsquare of $Q_0$, let $I_1(Q)$ and
$I_2(Q)$ be the squares with the same centre as $Q$ and having side length
$\lambda_1$ and $\lambda_2$ times the side length of $Q$, respectively, where
\[
0< \lambda_2<\lambda_1< 1.
\]
Note that $I_2(Q)$ is contained in the interior of $I_1(Q)$.

Given a line $\ell$, which is neither horizontal nor vertical, we define
\[
V_k^{(1)}(z) = \sharp \{Q\in \mathcal C_k : z\in \Pi_\ell(I_1(Q))\}
\]
and
\[
V_k^{(2)}(z) = \sharp \{Q\in \mathcal C_k : z\in \Pi_\ell(I_2(Q))\},
\]
where the number of elements in a set $A$ is denoted by $\sharp A$.
Let $\widetilde{V}_k^{(1,2)}$ be the version of the above, where $\Pi_\ell$ is
replaced by $\Pi_x$.

An observation in \cite{RS} is that if $p>\frac 12$ then for each $\ell$ one
can choose $\lambda_1$ and $\lambda_2$ such that for some $n$ and for all
$z\in \Pi_\ell(Q_0)$ we have
\[
\EE(V_{k+n}^{(2)}(z)) >  2 V_k^{(1)}(z).
\]
A similar statement can be obtained for $\widetilde{V}_k^{(i)}$, provided $x$
is sufficiently far away from $Q_0$ (the necessary distance depends on the
direction in which $x$ lies, and blows up for horizontal and vertical
directions. Note that the near-horizontal and near-vertical cases are dealt
with in Lemma~\ref{lem:mi1}.)

Similarly to the proof of Lemma~\ref{lem:mi1}, we can then check that if, for
some
finite family $\{z_i^{(k+1)n}\}$ of cardinality $K$, one has that
\begin{equation} \label{eqn:1}
V_{kn}^{(1)}(z_i^{((k+1)n)}) > M,
\end{equation}
then with probability $(1- (1-\Omega(1))^M)^K$,
\begin{equation} \label{eqn:2}
V_{(k+1)n}^{(2)}(z_i^{(k+1)n})> 2 M
\end{equation}
(and similarly for $\widetilde{V}_{kn}^{(1)}$, $\widetilde{V}_{(k+1)n}^{(2)}$).
With positive probability (e.g. corresponding to the probability that all
squares of level $n$ are chosen), the equation \eqref{eqn:1} is satisfied
for $k=1$ for all $z\in \Pi_\ell(Q_0(\frac 14))$ (resp.
$z\in\Pi_x(\underline Q_0(\frac 14))$for $\widetilde V^{(1)}_{kn}$).

As $I_2(Q)\subset I_1(Q)$, whenever $z$ belongs to the
projection of $I_2(Q)$, all $y$ close to $z$ belong to the projection of
$I_1(Q)$. Hence, if the implication \eqref{eqn:1} $\Rightarrow$
\eqref{eqn:2} holds for a finite family  $\{z_i^{(k+1)n}\}$ (of size $K$
increasing only exponentially fast with $k$), then
\begin{equation} \label{eqn:3}
V_{(k+1)n}^{(1)}(z)> 2 M
\end{equation}
for all $z\in\Pi_\ell(Q_0(\frac 14))$ (resp.
$z\in\Pi_x(\underline Q_0(\frac 14))$
for $\widetilde V^{(1)}_{(k+1)n}$). Note that the family $\{z_i^{(k+1)n}\}$
takes the place of the components in the proof Lemma~\ref{lem:mi1}.

An inductive argument completely analogous to the proof of Lemma~\ref{lem:mi1}
then allows us to conclude that
\[
\PP\left(  V_{kn}^{(1)}(z) \ge 2^k \textrm{ for all }z\in\Pi_\ell(Q_0(\frac 14))
\right) > 0,
\]
and likewise
\[
\PP\left(  \widetilde V_{kn}^{(1)}(z) \ge 2^k \textrm{ for all }
  z\in\Pi_\ell(\underline Q_0(\frac 14))\right) > 0.
\]
This finishes the proof.
\end{proof}

\end{document}